\documentclass{amsart}
\newcommand{\w}{\omega}
\newcommand{\e}{\varepsilon}
\newcommand{\F}{\mathcal F}
\newcommand{\A}{\mathcal A}
\newcommand{\IZ}{\mathbb Z}
\newcommand{\IN}{\mathbb N}
\newcommand{\IR}{\mathbb R}
\newcommand{\IT}{\mathbb T}
\newcommand{\IC}{\mathbb C}
\newcommand{\IQ}{\mathbb Q}
\newcommand{\conv}{\mathrm{conv}}
\newcommand{\pack}{\mathrm{pack}}
\newcommand{\cov}{\mathrm{cov}}
\newcommand{\supp}{\mathrm{supp}}
\newcommand{\is}{\mathsf{is}}
\newcommand{\Si}{\mathsf{Si}}
\newcommand{\us}{\mathsf{us}}
\newcommand{\iss}{\mathsf{iss}}
\newcommand{\uss}{\mathsf{uss}}
\newcommand{\sis}{\mathsf{sis}}
\newcommand{\his}{\widehat{\mathsf{\i s}}}

\newtheorem{theorem}{Theorem}[section]
\newtheorem{corollary}[theorem]{Corollary}
\newtheorem{proposition}[theorem]{Proposition}
\newtheorem{openprob}[theorem]{Open Problem}
\newtheorem{lemma}[theorem]{Lemma}
\newtheorem{problem}[theorem]{Problem}

\theoremstyle{definition}
\newtheorem{definition}[theorem]{Definition}
\newtheorem{fact}[theorem]{Fact}
\newtheorem{observation}[theorem]{Observation}
\newtheorem{question}[theorem]{Question}
\newtheorem{example}[theorem]{Example}
\newtheorem{remark}[theorem]{Remark}
\newtheorem{exercise}[theorem]{Exercise}

\title[Extremal density and measures on groups]{Extremal densities and measures on groups and $G$-spaces\\ and their combinatorial applications}
\author{Taras Banakh}

\begin{document}
\begin{abstract}
This text contains lecture notes of the course taught to Ph.D. students of Jagiellonian University in Krak\'ow on 25-28 November, 2013. The lecture course is based on the preprints \cite{Ban} and \cite{BPS}.
\end{abstract}

\maketitle

\section{Densities, submeasures and measures on sets}

Let $X$ be a set and $\mathcal P(X)$ be the Boolean algebra of all subsets of $X$. A function $\mu:P(X)\to[0,1]$ is called
\begin{itemize}
\item a {\em density} on $X$ if $\mu(\emptyset)=0$, $\mu(X)=1$ and $\mu$ is {\em monotone} in the sense that for any sets $A\subset B$ of $X$ we get $\mu(A)\le \mu(B)$;
\item a {\em submeasure} if $\mu$ is a {\em subadditive} density, which means that $\mu(A\cup B)\le\mu(A)+\mu(B)$ for any subsets $A,B\subset X$;
\item a {\em measure} on $X$ if $\mu$ is an {\em additive} density, i.e., $\mu(A\cup B)=\mu(A)+\mu(B)$ for any disjoint sets $A,B\subset X$.
\end{itemize}

\subsection{Some trivial examples of measures}

Any point $x$ of a set $X$ supports the {\em Dirac measure}
$\delta_x:\mathcal P(X)\to[0,1]$ defined by
$$\delta_x(A)=\begin{cases}0&\mbox{if $x\notin A$}\\
1&\mbox{if $x\in A$}.
\end{cases}
$$
A convex combination $\sum_{i=1}^n\alpha_i\delta_{x_i}$ of Dirac measures is called a {\em finitely supported measure}. The set of all finitely supported measures on $X$ will be denoted by $P_\w(X)$.
An infinite convex combination $\sum_{i=1}^\infty\alpha_i\delta_{x_i}$ of Dirac measures is countably supported measure on $X$.

It is natural to ask

\begin{question} Are there measures which are not countably supported?
\end{question}

The answer is affirmative and will be obtained by taking limits of finitely supported measures by ultrafilters.

\subsection{Filters and ultrafilters}

By a {\em filter} on a set $X$ we understand a family $\F$ of subsets of $X$ such that
\begin{itemize}
\item $\emptyset\notin\F$;
\item $A\cap B\in\mathcal F$ for any sets $A,B\in\F$;
\item $A\cup B\in\mathcal F$ for any sets $A\in\F$ and $B\subset X$.
\end{itemize}

\begin{example} For every $x\in X$ the family $\F_x=\{A\subset X:x\in A\}$ is a filter, called the {\em principal filter} generated by $x$.
\end{example}

\begin{observation} A filter $\F$ on a set $X$ is principal if and only if $\bigcap\F=\{x\}$ is a singleton.
\end{observation}

A filter $\F$ with empty intersection $\bigcap\F$ is called {\em free}.

\begin{example} For any infinite set $X$ the family $\F=\{X\setminus F:F\subset X$ is finite$\}$ is a free filter, called {\em the Fr\'echet filter} on $A$.
\end{example}

The family $Fil(X)$ of all filters on $X$ is partially ordered by the inclusion relation and by Zorn's Lemma has maximal elements, called  ultrafilters. More precisely, an {\em ultrafilter} is a filter $\F$ on $X$ which is not contained in any strictly larger filter.

\begin{example} Any principal filter $\mathcal F_x$ on $X$ is an ultrafilter, which can be called the {\em Dirac ultrafilter}.
\end{example}

Zorn's Lemma implies:

\begin{theorem}[Tarski] Each filter on a set $X$ is contained in some ultrafilter. In particular, the Fr\'echet filter on each infinite set can be enlarged to some free ultrafilter.
\end{theorem}

\begin{theorem}[Posp\'\i \v sil] On any infinite set $X$ there are $2^{2^{|X|}}$ (free) ultrafilters.
\end{theorem}

\begin{theorem}[Criterion of an ultrafilter] For a filter $\F$ on a set $X$ the following conditions are equivalent:
\begin{itemize}
\item $\F$ is an ultrafilter;
\item for any partition $X=A\cup B$ either $A$ or $B$ belongs to $\F$;
\item for any finite partition $X=A_1\cup\dots\cup A_n$ one of the sets $A_i$ belongs to $\F$;
\item for any set $F\in\F$ and  a partition $F=A\cup B$ either $A$ or $B$ belongs to $\F$.
\end{itemize}
\end{theorem}

\begin{proof} Exercise.
\end{proof}

\begin{remark} This theorem is crucial for applications of ultrafilters in Ramsey Theory, see \cite{HS}, \cite{Protbook}.
\end{remark}

\begin{fact} For any ultrafilter $\F$ on a set $X$ its characteristic function $\mu_\F:\mathcal P(X)\to\{0,1\}$ is a $\{0,1\}$-valued measure on $X$.
Conversely, for each $\{0,1\}$-valued measure $\mu:\mathcal P(X)\to\{0,1\}$ the family $\F=\mu^{-1}(1)$ is an ultrafilter on $X$.
\end{fact}

So, measures are more complicated objects than ultrafilters!

\subsection{Limits by filters and ultrafilters}

Let $A$ be a set, $\F$ be a filter on $A$, and $(x_\alpha)_{\alpha\in A}$ be a ``sequence'' in a topological space $X$. We shall say that a point $x$ is an {\em $\F$-limit} of the sequence $(x_\alpha)_{\alpha\in A}$ if for any neighborhood $O_x\subset X$ of $x$ the set $\{\alpha\in A:x_\alpha\in O_x\}$ belongs to the filter $\F$. In this case we shall also say that $(x_\alpha)_{\alpha\in A}$ is {\em $\F$-convergent} to $x$ and will denote the $\F$-limit $x$ by $\lim_{\alpha\to\F}x_\alpha$.

\begin{example} A sequence $(x_n)_{n\in\w}$ in a topological space $X$ converges to a point $x$ if and only if it is $\F$-convergent to $x$ for the Fr\'echet filter $\F$ on $\w$.
\end{example}

\begin{fact} In a Hausdorff topological space any $A$-sequence $(x_\alpha)_{\alpha\in A}$ has at most one $\F$-limit.
\end{fact}

\begin{theorem} For any ultrafilter $\F$ on a set $A$ any sequence $(x_\alpha)_{\alpha\in A}$ in a compact Hausdorff space $X$ converges to some (unique) point $x\in X$.
\end{theorem}

\begin{proof} Exercise.
\end{proof}

\begin{corollary} For any ultrafilter $\F$ on a set $A$ any bounded sequence $(x_\alpha)_{\alpha\in A}$ of real numbers has a unique $\F$-limit.
\end{corollary}

\begin{example} For any ultrafilter $\F$ on a set $A$ and any $A$-sequence $(\mu_\alpha)_{\alpha\in A}$ of measures on a set $X$ we can define the $\F$-limit measure $\lim_{\alpha\to\F}\mu_\alpha$ on $X$.
\end{example}

\section{Invariant densities and measures on groups and $G$-spaces}

Let $G$ be a group. A density $\mu:\mathcal P(G)\to[0,1]$ is called
\begin{itemize}
\item {\em left-invariant} if $\mu(xA)=\mu(A)$ for any $A\subset G$ and $x\in G$;
\item {\em right-invariant} if $\mu(Ay)=\mu(A)$ for any $A\subset G$ and $y\in G$;
\item {\em invariant} if $\mu(xAy)=\mu(A)$ for any $A\subset G$ and $x,y\in G$.
\end{itemize}

These notions are partial cases of $G$-invariant densities on $G$-spaces.

A {\em $G$-space} is a set $X$ endowed with an action $\cdot:G\times X\to X$, $\cdot:(g,x)\mapsto gx$, of a group $G$. The action should satisfy two axioms:
\begin{itemize}
\item $g(hx)=(gh)x$ for any $g,h\in G$ and $x\in X$;
\item $1_Gx=x$ for any $x\in X$.
\end{itemize}
Here $1_G$ stands for the neutral element of the group $G$.

Each group $G$ can be considered as a $G$-space $X=G$ endowed with the left action $\cdot:G\times X\to X$, $\cdot:(g,x)\mapsto gx$, of the group $G$ on itself.

A density $\mu:\mathcal P(X)\to[0,1]$ on a $G$-space $X$ is called {\em $G$-invariant} if $\mu(gA)=\mu(A)$ for any $g\in G$ and $A\subset X$.

\begin{example} For any finite group $G$ the finitely supported measure $\sum_{x\in G}\frac1{|G|}\delta_x$ is invariant.
\end{example}

\begin{exercise} Prove that a finitely (more generally, countably) supported
measure $\mu=\sum_i\alpha_i\delta_{x_i}$ on an infinite group $G$ is not invariant.
\end{exercise}

\begin{question} Is there any invariant measure on the group $\IZ$ of integers?
\end{question}

\begin{theorem}[Banach] The group $\IZ$ admits an invariant measure.
\end{theorem}

\begin{proof} Take any free ultrafilter $\F$ on $\IN$ and prove that the limit measure $\lim_{n\to\F}\frac1{2n+1}\sum_{k=-n}^n\delta_k$ is an invariant measure on $\IZ$.
\end{proof}

\begin{example} The free group with two generators has no invariant measure.
\end{example}

\begin{definition} A $G$-space $X$ admitting a $G$-invariant measure $\mu:\mathcal P(X)\to[0,1]$ is called {\em amenable}.
\end{definition}

In particular, a group $G$ is amenable if it admits a left-invariant measure $\mu:\mathcal P(G)\to[0,1]$.

So, $\IZ$ is amenable while $F_2$ is not.

\subsection{The F\o lner condition}

A $G$-space $X$ satisfies the {\em F\o lner condition} if for any finite set $F\subset G$ and any $\e>0$ there is a finite set $E\subset X$ such that $|FE\setminus E|<\e|E|$.

\begin{theorem} A $G$-space $X$ satisfying the F\o lner condition possesses a $G$-invariant measure.
\end{theorem}

\begin{proof} Consider the set $A=\{(F,\e):F\subset G$ is a finite set and $\e>0\}$. The set $A$ is partially order by the order $(F,\e)\le(F',\e')$ iff $F\subset F'$ and $\e'\le\e$. Take any ultrafilter $\F$ on $A$ containing all uper cones ${\uparrow}(F,\e)=\{(F',\e')\in A:(F,\e)\le(F',\e')\}$. By the F\o lner condition, for every $(F,\e)\in A$  there is a finite set $E_{(F,\e)}\subset G$ such that $|F\cdot E_{(F,\e)}\setminus E_{(F,\e)}|<\e|E(F,\e)|$. It can be shown that the limit measure
$$\mu=\lim_{(F,\e)\to\F}\frac1{|E_{(F,\e)}|}\sum_{x\in E_{(F,\e)}}\delta_x$$on $X$ is invariant.
\end{proof}

\begin{theorem}[F\o lner] A group $G$ is amenable if and only if it satisfies the F\o lner condition.
\end{theorem}

\begin{corollary} Each abelian group is amenable.
\end{corollary}

\begin{proof} Exercise.
\end{proof}

\begin{proposition} The class of amenable groups is local and is closed under taking subgroup, quotient groups and extensions.
\end{proposition}

\begin{proof} Use the F\o lner condition.
\end{proof}

\begin{theorem} A $G$-space $X$ is amenable if the group $G$ is amenable.
\end{theorem}

\begin{proof} Fix a left-invariant mesure $\mu:\mathcal P(G)\to[0,1]$. Fix any point $x\in X$ and
consider the map $\pi:G\to X$, $\pi:g\mapsto gx$. Observe that the measure $\nu:\mathcal P(X)\to[0,1]$ defined by $\nu(A)=\mu(\pi^{-1}(A))$ for $A\subset X$ is $G$-invariant.
\end{proof}

\subsection{The upper Banach density}

For a $G$-space $X$ by $P_G(X)$ we shall denote the set of all $G$-invariant measures on $G$. The set $P_G(X)$ is non-empty if and only if the $G$-space $X$ is amenable.

The function $d^*:\mathcal P(X)\to[0,1]$ defined by
$$d^*(A)=\sup_{\mu\in P_G(X)}\mu(A)$$is a subadditive density called the {\em upper Banach density} on $X$.

The upper Banach density $d^*$ plays a significant role in Ramsey Theory since many classical theorems related to partition have their density versions.

\begin{theorem}[Van der Waerden, 1927] For any partition $\IN=A_1\cup\dots\cup A_n$ one of the cells $A_i$ contains infinitely long arithmetic progressions.
\end{theorem}

\begin{theorem}[Szemer\'edi, 1975] Any subset $A\subset\IZ$ of positive upper Banach density $d^*(A)>0$ contains arbitrarily long arithmetic progressions.
\end{theorem}

\begin{theorem}[Gallai-Witt, 30ies] For any finite partition $\IZ^n=A_1\cup\dots\cup A_n$ there is a cell $A_i$ of the partition containing a homothetic copy $nF+\vec b$ of each finite set $F\subset \IZ^n$.
\end{theorem}

\begin{theorem}[Furstenberg-Katznelson, 1991] Any subset $A\subset \IZ^n$ of positive upper Banach density $d^*(A)>0$ contains a homothetic copy $nF+\vec b$ of each finite set $F\subset \IZ^n$.
\end{theorem}

\begin{theorem}[Green-Tao, 2008] The set $P$ of prime numbers contains arbitrarily long arithmetic progressions.
\end{theorem}

\begin{exercise} Prove that the set $P$ of prime numbers has zero upper Banach density $d^*(P)$ in the group $\IZ$.
\end{exercise}

\begin{problem} Can a counterpart of the upper Banach density be defined in arbitrary (not necessarily amenable) $G$-space?
\end{problem}

The answer is affirmative and will be given with help of extremal densities $\is_{12}$ and $\Si_{21}$ considered in the next lecture.

\section{Extremal densities $\is_{12}$ and $\Si_{12}$}

\subsection{Convolutions of measures}

For two finitely supported measures $\mu=\sum_i\alpha_i\delta_{x_i}$ and $\nu=\sum_j\beta_j\delta_{y_j}$ on a group $G$ their convolution is the measure $\mu*\nu=\sum_{i,j}\alpha_i\beta_j\delta_{x_iy_j}$ on $G$.

For Dirac measures $\delta_x$, $\delta_y$ the convolution $\delta_x*\delta_y=\delta_{xy}$, which means that the convolution is a binary (associative) operation on the set $P_\w(G)$ of finitely supported measures, extending the group operation on $G$.

For a $G$-space $X$ we can also define a convolution $\mu*\nu$ of a finitely supported measure $\mu=\sum_j\alpha_i\delta_{g_i}$ on $G$ and an arbitrary density $\nu:\mathcal P(X)\to[0,1]$ on $X$ letting
$\mu*\nu(A)=\sum_i\alpha_i\mu(g_i^{-1}A)$ for $A\subset X$. The density $\mu*\nu$ is (sub)additive if so is the density $\nu$.

\subsection{The extremal density $\is_{12}$}

The {\em extremal density} $\is_{12}:\mathcal P(X)\to[0,1]$ on a $G$-space $X$ is the density defined by the formula
$$\is_{12}(A)=\inf_{\mu_1\in P_\w(G)}\sup_{\mu_2\in P_\w(X)}\mu_1*\mu_2(A)$$for a subset $A\subset X$.

\begin{proposition} For any subset $A\subset X$ we get
$$\is_{12}(A)=\inf_{\mu\in P_\w(G)}\sup_{x\in X}\mu*\delta_x(A).$$
\end{proposition}

\begin{proof} Exercise.
\end{proof}

\begin{proposition} On any $G$-space $X$ the extremal density $\is_{12}$ is $G$-invariant. On any group $G$ the extremal density $\is_{12}$ is invariant under two sided shifts and also under automorphisms.
\end{proposition}

\begin{proof}
Exercise.
\end{proof}

\begin{exercise} Prove that $\is_{12}(2\IZ)=\frac12$ in the group $\IZ$.
\end{exercise}

\begin{theorem} For any amenable $G$-space $X$ the extremal density $\is_{12}$ is equal to the upper Banach density $d^*$.
\end{theorem}

The proof of this theorem will follow from a minimax characterization of $\is_{12}=\Si_{21}$ proved in a next section.

\subsection{The extremal density $\Si_{21}$}

The extremal density $\Si_{21}$ on a $G$-space $X$ is defined by
$$\Si_{21}(A)=\sup_{\mu_1\in P(X)}\inf_{\mu_2\in P_\w(G)}\mu_2*\mu_1.$$

\begin{proposition} $\Si_{21}(A)=\sup_{\mu\in P(X)}\inf_{x\in G}\mu(xA)$ for any subset $A\subset G$.
\end{proposition}

\begin{proof} Exercise.
 \end{proof}

We are going to prove that $\is_{12}=\Si_{21}$ on any $G$-space $X$.
For this we shall need some information on:

\subsection{Kelley intersection number}

For any family $\mathcal B$ of subsets of a set $X$ its {\em Kelley intersection number} is defined as
$$I(\mathcal B)=\inf_{B_1,\dots,B_n\in\mathcal B}\frac1n\max\{|F|:F\subset \{1,\dots,n\}\;
\bigcap_{k\in F}B_k\ne\emptyset\}.$$

\begin{exercise} Prove that $I(\mathcal B)=\inf\big\{\|f\|:f\in\conv(\{\chi_B:B\in\mathcal B\})\big\}$.
\end{exercise}

\begin{theorem}[Kelley, 1959] For any family $\mathcal B$ of subsets of a set $X$ its intersection number
 $$I(\mathcal B)=\sup_{\mu\in P(X)}\inf_{B\in \mathcal B}\mu(B).$$
\end{theorem}

\begin{proof} Apply the Hahn-Banach Theorem.
\end{proof}

\subsection{The equality $\is_{12}=\Si_{12}$}

\begin{theorem}\label{is=Kelley} For any subset $A$ of a $G$-space $X$ we get
 $$\is_{12}(A)=I(\{xA\}_{x\in G})=\Si_{21}(A).$$
\end{theorem}

\begin{proof} Kelley Theorem implies that $$I(\{xA\}_{x\in G})=\sup_{\mu\in P(X)}\inf_{x\in G}\mu(xA)=\sup_{\mu\in P(X)}\inf_{x\in G}\delta_{x^{-1}}*\mu(A)=\Si_{21}(A).$$

So, it remains to prove that $\is_{12}(A)=I(\{xA\}_{x\in G})$.

To prove that $\is_{12}(A)\le I(\{xA\})_{x\in G}$, it suffices to check that $\is_{12}(A)\le I(\{xA\}_{x\in G})+\e$ for every $\e>0$. The definition of $I(\{xA\}_{x\in G})$ yields points $x_1,\dots,x_n\in G$ such that $\|\frac1n\sum_{i=1}^n\chi_{x_iA}\|<I(\{xA\}_{x\in G})+\e$.

Consider the finitely supported measure $\mu=\frac1n\sum_{i=1}^n\delta_{x_i^{-1}}$ on $G$ and observe that for every $y\in X$ we get
$$\mu*\delta_y(A)=\frac1n\sum_{i=1}^n\delta_{x_i^{-1}y}(A)=\frac1n\sum_{i=1}^n\chi_{x_iA}(y)\le
\big\|\tfrac1n\sum_{i=1}^n\chi_{x_iA}\big\|<I(\{xA\}_{x\in G})+\e.$$
Consequently,
$$\is_{12}(A)\le\sup_{y\in G}\mu*\delta_y(A)\le I(\{xA\}_{x\in G})+\e.$$

Assuming that $\is_{12}(A)\ne I(\{xA\}_{x\in G})$, we can find $\e>0$ with $\is_{12}(A)<I(\{xA\}_{x\in G})-\e$. The definition of $\is_{12}$ yields a measure $\mu\in P_\w(G)$ such that $\sup_{y\in X}\mu*\delta_y(A)<I(\{xA\}_{x\in G})-\e$. Write $\mu$ as $\sum_{i}\alpha_i\delta_{z_i}$. Replacing the measure $\mu$ by a near measure, we can assume that all $\alpha_i$ are rational and have a common denominator $n$, in which case $\mu=\frac1n\sum_{i=1}^n\delta_{x_i}$ for some points $x_1,\dots,x_n$. Then
$$I(\{xA\}_{x\in G})\le\sup_{y\in X}\frac1n\sum_{i=1}^n\chi_{x_i^{-1}A}(y)=\sup_{y\in X}\frac1n\sum_{i=1}^n\delta_{x_iy}(A)=\sup_{y\in X}\mu*\delta_y(A)<I(\{xA\}_{x\in G})-\e$$and this is a desired contradiction implying that $\is_{12}(A)=I(\{xA\}_{x\in G})$.
\end{proof}

\subsection{The equality $d^*=\is_{12}=\Si_{21}$}

Recall that for a $G$-space $X$ by $P_G(X)$ we denote the family of all $G$-invariant measures on $X$.
This family is not empty if and only if $X$ is amenable.

\begin{theorem}\label{d=is} For any $G$-space $X$ endowed with an action of an amenable group $G$, we get
 $$\is_{12}(A)=\Si_{21}(A)=\sup_{\mu\in P_G(X)}\mu(A)=d^*(A).$$
\end{theorem}

\begin{proof} It is clear that $d^*(A)=\sup_{\mu\in P_G(X)}\mu(A)=\sup_{\mu\in P_G(X)}\inf_{x\in G}\mu(xA)\le\sup_{\mu\in P(X)}\inf_{x\in G}\mu(xA)=\Si_{21}(A)$.

It remains to show that $\Si_{21}(A)\le d^*(A)+\e$ for every $\e>0$. By the definition of $\Si_{21}(A)$, there is a measure $\mu\in P(X)$ such that $\inf_{x\in G}\mu(xA)>\Si_{21}(A)-\e$. Consider the Banach space $\ell_\infty(G)$ endowed with the action $G\times\ell_\infty(G)\to\ell_\infty(G)$, $(g\cdot\varphi)(x)=\varphi(g^{-1}x)$ of the group $G$. Observe that $g\cdot\chi_B=\chi_{gB}$ for any set $B\subset G$.

The group $G$, being amenable, admits a left-invariant measure. The integral by this measure is a positive $G$-invariant functional $a^*:\ell_\infty(G)\to\IR$ with unit norm. For every set $B\subset X$ consider the function $f_B:G\to\IR$, $f_B:x\mapsto \mu(x^{-1}B)$, which belongs to $\ell_\infty(G)$. Then $\nu(B)=a^*(f_B)$ is a required $G$-invariant measure on $X$ such that $\nu(A)=a^*(f_A)\ge\inf_{x\in G}\mu(xA)>\Si_{21}(A)-\e$. Since $\mu\in P_G(X)$, we get the required inequality
$$d^*(A)\ge\nu(A)>\Si_{12}(A)-\e.$$
\end{proof}

\begin{corollary} For any $G$-space $X$ endowed with an action of an amenable group $G$, the extremal densities $\is_{12}=\Si_{21}$ are subadditive.
\end{corollary}

\begin{question} Is the extremal desnity $\is_{12}$ subadditive on each amenable $G$-space?
\end{question}

\begin{example} On the free group $F_2$ the extremal density $\is_{12}$ is not subadditive.
\end{example}

Theorem~\ref{d=is} suggests to call the extremal density $\is_{12}=\Si_{12}$ the {\em upper Banach density} for each (not necessarily amenable) $G$-space $X$.

A subset $A$ of a $G$-space $X$ will be called
\begin{itemize}
\item {\em Banach null} if $\is_{12}(A)=0$;
\item {\em Banach positive} if $\is_{12}(A)>0$.
\end{itemize}

\subsection{Packing index of subsets in $G$-spaces}

For a set $A$ of a $G$-space $X$ its {\em packing index} $\pack(A)$ is defined as
$$\pack(A)=\sup\{|F|:F\subset G,\;\{xA\}_{x\in F}\mbox{ is disjoint}\}.$$

\begin{example} $\pack(2\IZ)=2$.
\end{example}

\begin{exercise} Show that the set $A$ of reduced words in the free group $F_2$ that start with $a$ or $a^{-1}$ has infinite packing index.
\end{exercise}

\begin{proposition}\label{pack-is} Any Banach positive subset $A$ of a $G$-space $X$ has finite packing index $$\pack(A)\le \frac1{\is_{12}(A)}.$$
\end{proposition}

\begin{proof} Assuming that $\pack(A)>1/\is_{12}(A)$, we can find a finite set $F\subset G$ of cardinality $|F|>1/\is_{12}(A)$ such that the family $(xA)_{x\in F}$ is disjoint. Consider the uniformly distributed measure $\mu=\frac1{|F|}\sum_{x\in F}\delta_{x^{-1}}$ on the group $G$ and observe that for every $y\in G$
$$\mu*\delta_y(A)=\frac1{|F|}\sum_{x\in F}\delta_{x^{-1}y}(A).$$
We claim that $\sum_{x\in F}\delta_{x^{-1}y}(A)\le 1$ for every $y\in G$. In the opposite case we can find two distinct points $u,v\in F$ such that $u^{-1}y,v^{-1}y\in A$. Then $u^{-1}v\in Ayy^{-1}A^{-1}=AA^{-1}$ and hence $uA\cap vA\ne\emptyset$, which contradicts the choice of $F\ni u,v$.
\end{proof}

The proof of this proposition yields a bit more, namely:

\begin{proposition}\label{pack<us} Any subet $A\subset G$ with positive density $\us_{12}(A)$ has finite packing index $$\pack(A)\le \frac1{\us_{12}(A)}.$$
\end{proposition}

Here $$\us_{12}=\inf_{\mu_1\in P_u(G)}\sup_{\mu_2\in P_\w(X)}\mu_1*\mu_2(A)$$where $P_u(G)$ is the family of all uniformly distributed finitely supported measures.

A finitely supported measure $\mu$ is {\em uniformly distributed} if $\mu=\frac1{|F|}\sum_{x\in F}\delta_x$ for some finite set $F\subset G$.

It is clear that $\is_{12}\le\us_{12}$.

\begin{theorem}[Solecki, 2005]
\begin{enumerate}
 \item For each amenable group $G$ we get $\is_{12}=\us_{12}$.
 \item If $G$ is countable and contains a subgroup isomorphic to the free group $F_2$, then for every $\e>0$ there is a set $A\subset G$ such that $\is_{12}(A)=0$ adn $\us_{12}(A)>1-\e$.
\end{enumerate}
\end{theorem}

\begin{observation} For any subset $A$ of a $G$-space $X$ the following conditions are equivalent:
 \begin{enumerate}
\item $\is_{12}(A)=1$;
\item $\us_{12}(A)=1$;
\item $A$ is {\em right thick} in the sense that for every finite set $F\subset G$ there is a point $x\in X$ such that $Fx\subset A$.
\end{enumerate}
\end{observation}

The equality $\is_{12}=\us_{12}$ proved by Solecki for amenable groups and the subadditivity of the upper Banach density $d^*$ imply:

\begin{corollary} On any amenable group $G$ the extremal densities $$\is_{12}=\us_{12}=\Si_{21}=d^*$$ are subadditive.
\end{corollary}

\begin{example} On the free group $F_2$ the extremal densities $\is_{12}$ and $\us_{12}$ are not subadditive.
\end{example}

\begin{proof} The group $F_2$ decomposes into the union $F_2=A\cup B\cup\{e\}$ of subsets with infinite packing indices and hence zero density $\is_{12}$ and $\us_{12}$.
 \end{proof}

\section{Some combinatorial applications of the densities $\is_{12}$ and $\us_{12}$}

\subsection{Motivating theorems of Banach-Kuratowski-Pettis and Steinhaus-Weil}

\begin{theorem}[Banach-Kuratowski-Pettis] For any non-meager Borel (more generally, analytic) subsets $A,B$ in a Polish group $G$ the set $AA^{-1}$ is a neighborhood of the unit in $G$ and the set $AB$ has non-empty interior in $G$.
 \end{theorem}

We recall that a topological space $A$ is called {\em analytic} if it is a continuous image of a Polish (= separable completely metrizable) space.

\begin{theorem}[Steinhaus-Weil] For any subsets $A,B\subset G$ of positive Haar measure in a compact topological group $G$ the set $AA^{-1}$ is a neighborhood of the unit in $G$ and the set $AB$ has non-empty interior in $G$, which implies that $G=FAA^{-1}$ and $G=FAB$ for some finite set $F\subset G$.
\end{theorem}

\begin{problem} What can be said about the sets $AA^{-1}$ and $AB$ for subsets $A,B\subset G$ with positive density $\is_{12}$ or $\us_{12}$?
\end{problem}

\subsection{Covering numbers of subsets in groups}

For a non-empty subset $A$ of a group $G$ the cardinal number
$$\cov(A)=\min\{|F|:F\subset G,\;\;G=FA\}$$ is called the {\em covering number} of the set $A$ in $G$.

\begin{example} $\cov(2\IZ)=2$ in $\IZ$ and $\cov(2\IZ)=\mathfrak c$ in $\IR$.
\end{example}

\begin{proposition}\label{cov-pack} Any subset $A$ of a group $G$ has $\cov(AA^{-1})\le\pack(A)$.
\end{proposition}

\begin{proof} By Zorn's Lemma, find a maximal set $F\subset G$ such that the family $(xA)_{x\in F}$ is disjoint. The maximality of $F$ guarantees that for every $y\in G$ the set $yA$ meets some set $xA$, $x\in F$. Consequently, $y\in FAA^{-1}$ and hence $\cov(AA^{-1})\le|F|\le\pack(A)$.
\end{proof}

\begin{corollary} If a set $A\subset G$ has positive density $\us_{12}(A)$, then $$\cov(AA^{-1})\le\pack(A)\le1/\us_{12}(A)$$ is finite.
\end{corollary}

\begin{corollary} If analytic subsets $A,B$ of a Polish group $G$ have positive upper Banach density $\us_{12}(A)$, then
$AA^{-1}$ is not meager and $AA^{-1}AA^{-1}$ is a neighborhood of the unit.
\end{corollary}

\subsection{Theorem of Jin-Beiglb\"ock-Bergelson-Fish-DiNasso-Lupini}

Recall that a subset $A$ of a group $G$ is called {\em Banach positive} if $\is_{12}(A)>0$.

\begin{theorem}[Jin-Beiglb\"ock-Bergelson-Fish-DiNasso-Lupini]\label{JBBF} For any Banach positive sets $A,B$ in an amenable group $G$ there is a finite set $F\subset G$ such that
$FAB$ is thick (equivalently, $\is_{12}(FAB)=1$).
\end{theorem}

To prove this theorem we shall need:

\subsection{An ergodicity property of the density $\is_{12}$}

Let us recall that a $G$-invariant measure $\mu:\mathcal B(X)\to[0,1]$ on a $G$-space $X$ is {\em ergodic} if any measurable $G$-invariant subset $A\subset X$ has measure $\mu(A)\in\{0,1\}$.

\begin{theorem}\label{ergodic} For any subset $A$ of a $G$-space $X$ we get
 $\sup\limits_{F\subset G}\is_{12}(FA)\in\{0,1\}$.
\end{theorem}

\begin{proof} It suffices to show that for any Banach positive subset $A\subset X$ and every $\e>0$ there is a finite set $F\subset G$ such that $\is_{12}(FA)>1-\e$. Find a positive number $\delta$ such
$\frac{\is_{12}(A)-\delta}{\is_{12}(A)+\delta}>1-\e$.
By Theorem~\ref{is=Kelley}, $\is_{12}(A)=I(\{xA\}_{x\in G})$. Then the definition of the intersection number yields points $x_1,\dots,x_n\in G$ such that $$\Big\|\frac1n\sum_{i=1}^n\chi_{x_iA}\Big\|<I(\{xA\}_{x\in G})+\delta=\is_{12}(A)+\delta$$and hence
\begin{equation}\label{eqq}
\frac1n\sum_{i=1}^n\chi_{x_iA}\le (\is_{12}(A)+\delta)\cdot\chi_{FA}
\end{equation}
where $F=\{x_1,\dots,x_n\}$. By the definition of the extremal density $\Si_{21}$, there is a measure $\mu$ on $X$ such that $\inf_{x\in G}\mu(xA)>\Si_{21}(A)-\delta=\is_{12}(A)-\delta$. Integrating the inequality (\ref{eqq}) by the measure $\mu$ we get
$$(\is_{12}(A)+\delta)\cdot \mu(FA)\ge \frac1n\sum_{i=1}^n\mu(x_iA)>\is_{12}(A)-\delta,$$which implies
the desired lower bound $$\mu(FA)>\frac{\is_{12}(A)-\delta}{\is_{12}(A)+\delta}>1-\e.$$
\end{proof}

\begin{lemma}\label{ergo-sum} Let $A,B$ be two subsets of an amenable group $G$. If $\us_{12}(A)+\us_{12}(B)>1$, then $\is_{12}(AB)=\us_{12}(AB)=1$.
\end{lemma}

\begin{proof} Choose a positive real number $\e>0$ such that $\us_{12}(A)+\us_{12}(B)>1+\e$. The equality $\is_{12}(AB)=\us_{12}(AB)=1$ will follow as soon as we check that for every finite subset $F\subset G$ there is a point $z\in G$ such that $Fz\subset AB$. We lose no generality assuming that $F$ contains the unit of the group $G$.

The amenability of $G$ yields a finite subset $E\subset G$ such that $|F^{-1}E\setminus E|<\e|E|$.
Since $\us_{12}(A)\le\max_{y\in G}\frac{|Ey\cap A|}{|E|}$, there is a point $y\in G$ such that $\frac{|Ey\cap A|}{|E|}\ge\us_{12}(A)$. Let $K=Ey$ and observe that $|F^{-1}K\setminus K|<\e|K|$ and $|K\cap A|\ge\us_{12}(A)|K|$. Then for every $x\in F$ we obtain that
$$
\begin{aligned}
\us_{12}(A)\cdot|K|&\le |K\cap A|\le|(xK\cup (K\setminus xK))\cap A|\le |xK\cap A|+|K\setminus xK|=\\
&=|K\cap x^{-1}A|+|x^{-1}K\setminus K|\le |K\cap x^{-1}A|+|F^{-1}K\setminus K|<|K\cap x^{-1}A|+\e|K|,
\end{aligned}
$$
and hence  $|K\cap x^{-1}A|>(\us_{12}(A)-\e)\cdot|K|$.

Since $\us_{12}(B)\le\max_{z\in G}\frac{|K^{-1}\cap Bz^{-1}|}{|K^{-1}|}$, there is a point $z\in G$ such that $\frac{|K^{-1}\cap Bz^{-1}|}{|K|}\ge \us_{12}(B)$.
Observe that for every point $x\in F$
$$|K\cap x^{-1}A|+|K\cap zB^{-1}|=|K\cap x^{-1}A|+|K^{-1}\cap Bz^{-1}|>(\us_{12}(A)-\e)\cdot|K|+\us_{12}(B)\cdot|K|>|K|,$$which implies that the set $K\cap x^{-1}A$ and $K\cap zB^{-1}$ have a common point and hence $xz\in AB$ and $Fz\subset AB$.
\end{proof}
\smallskip

Now we are able to present
\smallskip

\noindent{\em Proof of Jin-Beiglb\"ock-Bergelson-Fish-DiNasso-Lupini Theorem}: Let $A,B$ be two Banach positive sets in an amenable group $G$. By the Ergodic Theorem~\ref{ergodic}, there is a finite subset $F\subset G$ such that $\is_{12}(FA)>1-\is_{12}(B)$. Then $\us_{12}(FA)+\us_{12}(B)\ge \is_{12}(FA)+\is_{12}(B)>1$ and $\is_{12}(FAB)=1$ for some finite set $F\subset G$.\qed
\medskip

By methods of non-standard analysis, Di Nasso and Lupini \cite{NL} proved the following improvement of Theorem~\ref{JBBF}:

\begin{theorem}[Di Nasso, Lupini] Any Banach positive sets $A,B$ in an amenable group $G$ have $\cov(AB)\le\ 1/(d^*(A)\cdot d^*(B))$.
\end{theorem}

\subsection{A motivating problem of Ellis}

The following old problem of Ellis was motivated by the mentioned theorems of Banach-Kuratowski-Pettis and Steinhaus-Weil.

\begin{openprob}[Ellis, 50ies] Is it true that for each subset $A\subset\IZ$ with $\cov(A)<\infty$ the difference $A-A$ is a neighborhood of zero in the Bohr topology on $\IZ$?
\end{openprob}

\subsection{The Bohr topology on groups}

The {\em Bohr topology} on a group $G$ is the largest totally bounded group topology on $G$. This topology is not necessarily Hausdorff. A subset $A\subset G$ of a group $G$ will be called {\em Bohr-open} (resp. {\em Bohr-closed}) if it is open (resp. closed) in the Bohr topology on $G$.

The Bohr-closure $\bar 1_G$ of the singleton $\{1_G\}$
is a Bohr-closed normal subgroup of $G$ and the quotient group $G/\bar 1_G$ endowed with the quotient Bohr topology is a totally bounded Hausdorff topological group whose Raikov completion is a Hausdorff compact topological group. This compact topological group is denoted by $bG$ and called the {\em Bohr compactification} of $G$. The quotient homomorphism $i:G\to bG$ is called the {\em canonical homomorphism} from $G$ to $bG$.

\begin{remark} For any abelian group $G$ the canonical homomorphism $i:G\to bG$ is injective.
 On the other hand, there are non-commutative amenable groups $G$ with trivial Bohr compactification. A simple example of such a group is the group $A_X$ of even finitely supported permutations of an infinite set $X$. The group $A_X$ is locally finite and hence amenable.
\end{remark}

\subsection{A Theorem of Bogoliuboff-F\o lner-Cotlar-Ricabarra-Ellis-Keynes-Beiglb\"ock-Bergelson-Fish-Banakh}

The proof of the following theorem can be found in \cite{Ban}.

\begin{theorem}[Bogoliuboff-F\o lner-Cotlar-Ricabarra-Ellis-Keynes-Beiglb\"ock-Bergelson-Fish-Banakh]\label{t:Bohr} For any Banach positive set $A$ in an amenable group $G$ the difference set $AA^{-1}=U\setminus N$ for some Bohr-open neighborhood $U$ of $1_G$ and some Banach null set $N\subset G$.
\end{theorem}

\begin{corollary} For any amenable group $G$ and Banach positive subsets $A,B\subset G$, the set $B^{-1}AA^{-1}$ has non-empty interior in the Bohr topology and $AA^{-1}BB^{-1}$ is a neighborhood of the unit $1_G$ in the Bohr topology on $G$.
\end{corollary}

\begin{proof} 1. By Theorem~\ref{t:Bohr}, there are a Bohr open neighborhood $U\subset G$ of the unit and a  Banach null set $N\subset G$ such that $U\setminus N\subset AA^{-1}$. Since the multiplication and the inversion are continuous in the Bohr topology on $G$, there is a Bohr open neighborhood $V\subset G$ of the unit such that $VV^{-1}\subset U$. By the total boundedness of the Bohr topology, there is a finite subset $F\subset G$ such that $G=VF$. Since $B=\bigcup_{x\in F}Vx\cap B$, the subadditivity of the upper Banach density $\is_{12}$ on the amenable group $G$ yields a point $x\in F$ such that $B_x=Vx\cap B$ has positive upper Banach density $\is_{12}(B_x)$. We claim that $x^{-1}V\subset B_x^{-1}(U\setminus N)$. Given any point $v\in V$, consider the set $B_xx^{-1}v\subset Vxx^{-1}v\subset VV\subset U$. Since $\is_{12}(B_x)>0$, the set $B_xx^{-1}v$ is not contained in the Banach null set $N$ and hence meets the complement $U\setminus N$. Then $x^{-1}v\in B_x^{-1}(U\setminus N)\subset B^{-1}AA^{-1}$ and hence the set $B^{-1}AA^{-1}$ contains the non-empty Bohr open set $x^{-1}V$.
\smallskip

2. By Theorem~\ref{t:Bohr}, there are a Banach null sets $N_A,N_B\subset G$ and a Bohr open neighborhood $U\subset G$ of the unit such that $U\setminus N_A\subset AA^{-1}$ and $U\setminus N_B\subset BB^{-1}$. Using the continuity of the multiplication and inversion with respect to the Bohr topology on $G$, find a Bohr open neighborhood $V\subset G$ of the unit $1_G$ such that $VV^{-1}\subset U$. We claim that $V\subset AA^{-1}BB^{-1}$. The subadditivity of the ipper Banach density $d^*=\is_{12}$ and the total boundedness of the topological group $G$ imply that the neighborhood $V$ is not Banach null. The subadditivity of the upper Banach density $\is_{12}$ implies that $\is_{12}(V\setminus N_B)=\is_{12}(V)>0$. Then for every $v\in V$ the set $v(V\setminus N_B)\subset U$, being Banach positive, meets the set $U\setminus N_A$, which implies $v\in (U\setminus N_A)(V\setminus N_B)^{-1}\subset AA^{-1}BB^{-1}$.
\end{proof}

\begin{corollary} An amenable group $G$ has trivial Bohr compactification if and only if for any partition $G=A_1\cup\dots \cup A_n$ there is a cell $A_i$ of the partition such that $G=A_i^{-1}A_iA_i^{-1}$.
\end{corollary}

\subsection{Another theorem of Beiglb\"ock-Bergelson-Fish}

Let $G$ be a group. We shall say that a set $A\subset G$ is {\em finitely representable} in a set $B\subset G$ if for any finite subset $F\subset A$ some right shift $Fy$ is contained in $B$. In particular, a set $B$ is right-thick if and only if $G$ is finitely representable in $B$.

If $A$ is finitely representable in $B$, then $AA^{-1}\subset BB^{-1}$.

The following theorem for countable amenable groups was proved in \cite{BBF} and for all amenable groups in \cite{Ban}.

\begin{theorem}[Beiglb\"ock-Bergelson-Fisher-Banakh] For any subsets $A,B$ of positive upper Banach density in an amenable group $G$ some non-empty Bohr-open set is finitely representable in the set $AB$. Consequently, $ABB^{-1}A^{-1}$ is a neighborhood of the unit in the Bohr topology on $G$.
\end{theorem}

\subsection{A Motivating Problem of Protasov}

The following problem was posed by I.V.Protasov in an early version of the Kourovka Notebook \cite{Kourov}. The paper \cite{BPS} gives a survey of partial solutions of this problem.

\begin{problem}[Protasov] Is it true that for any partition $G=A_1\cup\dots\cup A_n$ of a group $G$ one cell $A_i$ of the partition has $\cov(A_iA_i^{-1})\le n$?
\end{problem}

Protasov's problem has the following partial solution proved in \cite{BRS}:

\begin{theorem}[Banakh-Ravsky-Slobodianiuk]\label{slob} For any partition $G=A_1\cup\dots\cup A_n$ of a group $G$ there is a set $F\subset G$ of cardinality $|F|\le\max_{1<k\le n}\frac{k^{n+1-k}-1}{k-1}\le n!$ such that $G=FA_iA_i^{-1}$ for some $i$.
\end{theorem}

For amenable groups the answer to Protasov's Problem is affirmative:

\begin{theorem} If a group $G$ is amenable (more generally, if the density $\is_{12}$ is subaditive), then
for any partition $G=A\cup\dots\cup A_n$ there is a cell $A_i$ with $\cov(A_iA_i^{-1})\le \pack(A)\le n$.
\end{theorem}

\begin{proof} By the subadditivity of $\is_{12}$ some cell $A_i$ has upper Banach density $\is_{12}(A_i)\ge1/n$ and then $\cov(A_iA_i^{-1})\le\pack(A_i)\le n$.
\end{proof}

\begin{example} The free group $F_2$ admits a partition $F_2=A_1\cup A_2$ such that $\pack(A_i)=\pack_2(A_i)=\w$ but $\cov(A_iA_i^{-1})\le\cov(A_i)\le 2$ for every $i\in\{1,2\}$.
\end{example}

\begin{problem}[Solecki] Is a group $G$ amenable if the density $\is_{12}$ is subadditive?
\end{problem}

The following theorem (proved in \cite{Ban}) shows that groups with subadditive density $\is_{12}$ are close to being amenable.

\begin{theorem} For a group $G$ the following conditions are equivalent:
 \begin{enumerate}
  \item the group $G$ is amenbale;
\item on the group $G\times\IZ$ the density $\is_{12}$ is subadditive;
\item for every $n\in\IN$ there is a finite group $F$ of cardinality $|F|\ge n$ such that on the group $G\times F$ the density $\is_{12}$ is subadditive.
 \end{enumerate}
\end{theorem}

\section{The extremal submeasure $\iss_{213}$}

On each $G$-space $X$ the extremal density $\iss_{213}:\mathcal P(X)\to[0,1]$ is defined for $A\subset X$ by
$$\iss_{213}(A)=\inf_{\mu_1\in P_\w(G)}\sup_{\mu_2\in P_\w(G)}\sup_{\mu_3\in P_\w(X)}\mu_2*\mu_1*\mu_3(A).$$

The extremal density $\iss_{213}$ will be called the {\em Solecki submeasure} on the $G$-space $X$.

\begin{proposition}$\is_{213}(A)=\inf_{\mu\in P_\w(G)}\sup_{x\in G}\sup_{y\in X}\mu*\delta_y(xA)$.
\end{proposition}

This implies that $\is_{12}\le \iss_{213}$.

\begin{theorem} On each $G$-space $X$ the density $\iss_{213}$ is a $G$-invariant submeasure.
\end{theorem}

\begin{proof} The $G$-invariance of $\iss_{213}$ follows from the definition.
The subadditivity of $\iss_{213}$ will follow as soon as we check that $$\iss_{213}(A\cup B)\le \iss_{213}(A)+\iss_{213}(B)+2\e$$ for any $\e>0$ and any subsets $A,B\subset X$.

The definition of $\iss_{213}(A)$ yields a measure $\mu_1\in P_\w(G)$ such that $\sup_{\mu_2\in P_\w(G)}\sup_{\mu_3\in P_\w(X)}\mu_2*\mu_1*\mu_3(A)<\iss_{213}(A)+\e$.
By analogy, there is a measure $\nu_1\in P_\w(G)$ such that $\sup_{\nu_2\in P_\w(G)}\sup_{\mu_3\in P_\w(X)}\nu_2*\nu_1*\nu_3(A)<\iss_{213}(A)+\e$. Then for the measure $\eta_1=\mu_1*\nu_1\in P_\w(G)$ and any measures $\eta_2\in P_\w(G)$, $\eta_3\in P_\w(X)$ we get $$\eta_2*\eta_1*\eta_3(A)=\eta_2*\mu_1*(\nu_1*\eta_3)(A)<\iss_{213}(A)+\e$$and
$$\eta_2*\eta_1*\eta_3(B)=(\eta_1*\mu_1)*\nu_2*\eta_3(B)<\iss_{213}(B)+\e.$$
Then
$$\eta_2*\eta_1*\eta_3(A\cup B)\le \eta_2*\eta_1*\eta_3(A)+\eta_2*\eta_1*\eta_3(B)<\iss_{213}(A)+\iss_{213}(B)+2\e$$
and hence
$$\iss_{213}(A\cup B)\le\sup_{\eta_2\in P_\w(G)}\sup_{\eta_3\in P_\w(X)}\eta_2*\eta_1*\eta_3(A\cup B)\le \iss_{213}(A)+\iss_{213}(B)+2\e.$$
\end{proof}

The proof of this theorem yields a bite more, namely:

\begin{theorem}\label{ishat} $\is_{12}(A\cup B)\le\is_{12}(A)+\iss_{213}(B)$ for any sets $A,B\subset X$.
\end{theorem}

\begin{theorem}[Solecki, 2005] On any group $G$ we get the equality $\iss_{213}=\uss_{213}$, where
 $$\uss_{213}(A)=\inf_{\mu_1\in P_u(G)}\sup_{\mu_2\in P_\w(G)}\sup_{\mu_3\in P_\w(G)}\mu_2*\mu_1*\mu_3(A)=\inf_{F\subset G}\max_{x,y\in G}\frac{|F\cap xAy|}{|F|}.$$
\end{theorem}

This theorem implies the inequality $\is_{12}\le\us_{12}=\uss_{213}=\iss_{213}$ for any group $G$.

\begin{question} Is $\uss_{213}=\iss_{213}$ for any $G$-space?
\end{question}

\subsection{The subadditivization $\widehat\mu$ of a density $\mu$}

For any density $\mu:\mathcal P(X)\to[0,1]$ on a set $X$ its {\em subadditivization} is the submeasure $\widehat\mu:\mathcal P(X)\to[0,1]$ defined by
$$\widehat\mu(A)=\sup_{C\subset X}\mu(A\cup C)-\mu(C).$$
It is clear that $\mu\le\widehat\mu$. To see that $\widehat\mu$ is subadditive, observe that for any sets $A,B,C\subset G$ we get
$$\mu(A\cup B\cup C)-\mu(C)=\mu(A\cup B\cup C)-\mu(B\cup C)+\mu(B\cup C)-\mu(C)\le \widehat\mu(A)+\widehat\mu(B).$$

The inequality $\is_{12}(A\cup B)\le\is_{12}(A)+\iss_{213}(B)$ proved in Theorem~\ref{ishat} implies

\begin{proposition} On any $G$-space $\is_{12}\le \his_{12}\le \iss_{213}$.
\end{proposition}

\subsection{An application of $\iss_{213}$ to Protasov's problem}

A subset $A$ of a group $G$ will be called {\em inner-invariant} if $xAx^{-1}=A$ for any $x\in G$.

\begin{observation} Any inner-invariant set $A\subset G$ has $\is_{12}(A)=\iss_{213}(A)$.
\end{observation}

\begin{theorem}[Banakh-Protasov-Slobodianiuk] For any partition $G=A_1\cup\dots\cup G_n$ of a group $G$ into inner-invariant sets one of the sets $A_i$ has $\cov(A_iA_i^{-1})\le\pack(A_i)\le n$.
\end{theorem}

\begin{proof} By the subadditivity of the Solecki submasure, one of the sets $A_i$ has submeasure $\iss_{213}(A_i)\ge 1/n$. Since $A_i$ is inner-invariant, $\is_{12}(A_i)=\iss_{213}(A_i)\ge1/n$ and then
 $$\cov(A_iA_i^{-1})\le\pack(A_i)\le\frac1{\is_{12}(A_i)}\le n.$$according to Propositions~\ref{cov-pack} and \ref{pack-is}.
\end{proof}

\subsection{Generalizing Gallai-Witt Theorem with help of $\iss_{213}$}

The submeasure $\iss_{213}$ can be used to generalize classical Gallai-Witt theorem to arbitrary groups.

\begin{theorem}[Gallai-Witt] For any partition $\IZ^n=A_1\cup\dots\cup A_n$ there is a cell $A_i$ of the partition containing homothetic copy $nF+b$ of any finite set $F\subset G$.
\end{theorem}

By a {\em homothety} on a group $G$ we shall undertstand any map $h:G\to G$ of the form $h(x)=a_0xa_1xa_2\dots xa_n$ for some constants $a_0,\dots,a_n\in G$. If $n=1$, then $h(x)=a_0xa_1$ is a two-sided shift.

Observe that any homothety in an abelian group $G$ can be written as $h(x)=nx+a$ for some $n\in\IN$ and $a\in G$.

\begin{theorem}\label{homo} If a subset $A$ of a group $G$ has
 \begin{enumerate}
  \item $\iss_{213}=1$, then for any finite subset $F\subset G$ there is a two-sided shift $h:G\to G$ such that $h(F)\subset A$;
\item $\iss_{213}>0$, then for any finite subset $F\subset G$ there is a homothety $h:G\to G$ with $h(F)\subset A$
 \end{enumerate}
\end{theorem}

In the proof we shall use the following theorem (implying the absence of draws in higher-dimensional tic-tac-toe).

\begin{theorem}[Furstenberg-Katznelson, 1991]For any finite set $F$ and a positive $\e$ there is a number $n\in\IN$ such that any subset $S\subset F^n$ of cardinality $|S|>\e|F^n|$ contains a ``line'', i.e., the image of $F$ under an injective map $\xi=(\xi_i)_{i=1}^n:F\to F^n$, whose components $\xi_i:F\to F$ are ether constant or identity maps.
\end{theorem}

\begin{proof}[Proof of Theorem~\ref{homo}]
1. The first item follows directly from the definition of $\iss_{213}(A)=1$.

2. Assume that $\e=\iss_{213}(A)>0$ and let $F\subset G$ be a non-empty finite subset of $G$. Let $n$ be the number given by Furstenberg-Katznelson Theorem. On the cube $F^n$ consider the uniformly distributed measure $\mu=\frac1{|F^n|}\sum_{x\in F^n}\delta_x$. The multiplication function $\pi:F^n\to G$, $\pi:(x_1,\dots,x_n)\mapsto x_1\cdots x_n$, maps the measure $\mu$ to some measure $\nu=\frac1{|F^n|}\sum_{x\in F^n}\delta_{\pi(x)}$ on $G$. Since $\e=\iss_{213}(A)\le\max_{x,y\in G}\nu(xAy)$, there are points $u,v\in G$ such that $\nu(uAv)\ge \e$. Then for the map $\pi_{u,v}:G\to G$, $\pi_{u,v}:\vec x\mapsto u^{-1}\cdot\pi(\vec x)\cdot v^{-1}$, the preimage $S=\pi_{u,v}^{-1}(A)$ has measure $\mu(S)=\nu(uAv)\ge\e$ and hence $|S|=\mu(S)\cdot|F^n|\ge\e|F^n|$. By the choice of $\e$, the set $S$ contains the image of $F$ under some injective function $\xi=(\xi_i)_{i=1}^n:F\to F^n$ whose components $\xi_i:F\to F$ are constant or the identity function. It follows that $h=\pi_{u,v}\circ\xi:F\to G$ is the restriction of some homothety on $G$ and $h(F)\subset \pi_{u,v}(S)\subset A$.
\end{proof}

\subsection{A strange property of $\iss_{213}$}

The inequality $\pack(A)\le1/\is_{12}(A)$ proved in Proposition~\ref{pack-is} implies that each subset $A$ with infinite packing index in a group $G$ has zero density $\is_{12}(A)=0$. For the submeasure $\iss_{213}$ it is not true as $F_2=A\cup B$ can be written as the union of two sets with infinite packing indices.

Moreover, we have the following strange

\begin{example}\label{strange} For any infinite cardinal $\kappa$ there is an amenable group $G$ of cardinality $|G|=\kappa$ containing a countable subgroup $H\subset G$ with $\iss_{213}(H)=1$.
\end{example}

\begin{proof} Let $G$ be the group of bijections $f:\kappa\to\kappa$ having finite support
$$\supp(f)=\{x\in\kappa:f(x)\ne x\},$$ and $H=\{f\in G:\supp(f)\subset\w\}$. Observe that for any finite set $\F\subset H$ the set $F=\bigcup_{f\in\F}\supp(f)$ is finite. Consequently, there is a bijection $g\in G$ such that $g(F)\subset\w\subset\kappa$. Then the set $g\circ \F\circ g^{-1}$ is contained in the countable subgroup $H$, which implies that $\iss_{213}(H)=1$.
\end{proof}

Something like this cannot happen in compact topological groups.

\begin{theorem} Each countable set $A$ in an infinite compact topological group $G$ has Solecki submeasure $\iss_{213}(A)=0$.
\end{theorem}

In fact, the proof of this theorem (from \cite{Ban}) yields more:

\begin{theorem} Each subset $A$ of cardinality $|A|<\cov(\mathcal E)$ in an infinite compact topological group $G$ has Solecki submeasure $\iss_{213}(A)=0$.
\end{theorem}

Here $\cov(\mathcal E)$ stands for the smallest cardinality of a cover of the real line by closed subsets of Lebegue measure zero. It is clear that $\cov(\mathcal E)\in[\aleph_1,\mathfrak c]$.
Martin's Axiom implies that $\cov(\mathcal E)=\mathfrak c$.

\begin{openprob} Is $\iss_{213}(A)=0$ for any subset $A$ of cardinality $|A|<|G|$ in an infinite compact topological group $G$?
 \end{openprob}

The answer to this problem is affirmative for countable groups $G$, see \cite{Ban}.

\section{The Solecki submeasure in compact topological groups}

\subsection{Solecky submeasure $\iss_{213}$ versus Haar submeasure $\bar\lambda$}

\begin{theorem}[Haar] Each compact topological group $G$ possesses a Haar measure $\lambda$, i.e., a unique invariant regular $\sigma$-additive probability measure $\lambda:\mathcal B(G)\to[0,1]$ defined on the $\sigma$-algebra $\mathcal B(G)$ of Borel subsets of $G$.
\end{theorem}

The Haar measure $\lambda$ induces the Haar submeasure $\bar \lambda:\mathcal P(G)\to[0,1]$ defined by $\bar\lambda(A)=\lambda(\bar A)$.

The Haar submeasure $\bar \lambda:\mathcal P(G)\to[0,1]$ can be defined on each group $G$ with help of the Bohr compactification $bG$
letting $$\bar\lambda(A)=\lambda(\overline{i(A)})$$ where $\lambda$ is the Haar measure on the Bohr compactification $bG$ of $G$ and $\overline{i(A)}$ is the closure of the canonical image $i(A)$ of $A$ in $bG$.

\begin{theorem}\label{Bohr} $\iss_{213}\le\bar\lambda$.
\end{theorem}

\begin{proof} Let $(bG,i)$ be the Bohr compactification of $G$ and $B$ be the closure of the set $i(A)$ in $bG$.

To prove the theorem, it suffices to check that $\sigma(A)\le\lambda(B)+\e$ for every $\e>0$. By the regularity of the Haar measure $\lambda$ and the normality of the compact Hausdorff space $bG$, the closed set $B$ has a closed neighborhood $\bar O(B)$ in $bG$ such that $\lambda(\bar O(B))<\lambda(B)+\e$. Let $1_{bG}$ denote the unit of the group $bG$. Since $1_{bG}\cdot B\cdot 1_{bG}=B\subset \bar O(B)$, the compactness of $B$ and the continuity of the group operation yield an open neighborhood $V\subset bG$ of $1_{bG}$ such that $VBV\subset \bar O(B)$. Then $\overline{VBV}\subset \bar O(B)$ and hence $\lambda(x\overline{VBV}y)=\lambda(\overline{VBV})\le\lambda(\bar O(B))<\lambda(B)+\e$ for any points $x,y\in bG$. The density of $i(G)$ in $bG$ implies that  $bG=\bigcup_{x\in i(G)}xV=\bigcup_{x\in i(G)}Vx$. By the compactness of $bG$ there is a finite set $F\subset i(G)$ such that $G=FV=VF$.

Let $P_\sigma(G)$ be the space of all probability regular Borel $\sigma$-additive measures on $G$ endowed with the topology generated by the subbase consisting of the sets $\{\mu\in P_\sigma(G):\mu(U)>a\}$ where $U$ is an open subset in $G$ and $a\in\mathbb R$. It follows that for each closed set $C\subset G$ the set $$\{\mu\in P_\sigma(G):\mu(C)<a\}=\{\mu\in P_\sigma(G):\mu(G\setminus C)>1-a\}$$is open in $P_\sigma(G)$.
Consequently, the set $$O_\lambda=\bigcap_{x,y\in F}\{\mu\in P_\sigma(G):\mu(x\overline{VBV}y)<\lambda(B)+\e\}$$is an open neighborhood of the Haar measure $\lambda$ in the space $P_\sigma(G)$.

Since $i(G)$ is a dense subset in $bG$, the subspace $P_\w(i(G))$ of finitely supported probability measures on $i(G)$ is dense in the space $P_\sigma(bG)$. Consequently, the open set $O_\lambda$ contains some probability measure $\mu\in P_\w(i(G))$ and we can find a finitely supported probability measure $\nu$ on $G$ such that $i(\nu)=\mu$. The latter equality means that $\mu(C)=\nu(i^{-1}(C))$ for all $C\subset bG$ and hence $\nu(D)\le\nu\big(i^{-1}(i(D))\big)=\mu(i(D))$ for each set $D\subset G$. We claim that $\sup_{x,y\in G}\nu(xAy)\le \iss_{213}(A)+\e$. Indeed, since $bG=FV=VF$, for any points $x,y\in G$ we can find points $x',y'\in F$ such that $i(x)\in x'V$ and $i(y)=Vy'$. Then
$$\nu(xAy)\le \mu(i(x)i(A)i(y))\le\mu(i(x)Bi(y))\le \mu(x'VBVy')\le \mu(x'\overline{VBV}y')<\lambda(B)+\e=\bar\lambda(A)+\e$$as $\mu\in O_\lambda$.
So, $\is_{213}(A)\le\sup_{x,y\in G}\nu(xAy)\le\bar\lambda(A)+\e$. Since the number $\e>0$ was arbitrary, we conclude that $\iss_{213}(A)\le\bar\lambda(A)$.
\end{proof}

\subsection{The Solecki submeasure $\iss_{213}$ versus Haar measure $\lambda$}

Let $G$ be a compact Hausdorff topological group and $\lambda$ be the Haar measure on $G$.

For a subset $A\subset G$ by $\bar A$ and $A^\circ$ we shall denote the closure and interior of $A$ in $G$, respectively.

\begin{proposition} Any subset $A\subset G$ has
 $\lambda(A^\circ)\le\his_{12}(A)\le\iss_{213}(A)\le \lambda(\bar A).$
\end{proposition}

\begin{proof} Theorem~\ref{Bohr} implies that $\iss_{213}(A)\le\lambda(\bar A)$.  Since the set $G\setminus A^\circ$ is closed in $G$, we get $$\his_{12}(G\setminus A^\circ)\le\iss_{213}(G\setminus A^\circ)\le\lambda(G\setminus A^\circ)=1-\lambda(A^\circ).$$
By the subadditivity of the submeasure $\his_{12}$, $$1=\his_{12}(G)\le\his_{12}(A^\circ)+\his_{12}(G\setminus A^\circ)\le \his_{12}(A^\circ)+1-\lambda(A^\circ),$$
which implies that $\lambda(A^\circ)\le\his_{12}(A)$.
\end{proof}

\begin{corollary}\label{zeroboundary} $\his_{12}(A)=\iss_{213}(A)=\lambda(A)$ for any subset $A\subset G$ whose boundary $\partial A=\bar A\setminus A^\circ$ has zero Lebesgue measure $\lambda(\partial A)=0$.
\end{corollary}

For the Solecki submeasure $\iss_{213}$ we can prove more:

\begin{lemma}\label{closed} $\iss_{213}(A)=\lambda(A)$ for each closed subset $A\subset G$.
\end{lemma}

\begin{proof} By Proposition~\ref{Bohr}, $\iss_{213}(A)\le\lambda(A)$. So, it remains to show that $\iss_{213}(A)\ge\lambda(A)$. Assuming conversely that $\iss_{213}(A)<\lambda(A)$ we conclude that the number $\e=\frac12(\lambda(A)-\iss_{213}(A))$ is positive. Then $\iss_{213}(A)<\lambda(A)-\e$ and by the definition of the Solecki submeasure, there is a finitely supported probability measure $\mu$ on $G$ such that $\sup_{x,y\in G}\mu(xAy)<\lambda(A)-\e$. For each pair $(x,y)\in G\times G$, by the regularity of the measure $\mu$,  there is an open neighborhood $O_{x,y}(A)\subset G$ of $A$ such that $\mu(xO_{x,y}(A)y)<\lambda(A)-\e$. Using the compactness of $A$, we can find an open neighborhood $U_{x,y}\subset G$ of $1_G$ such that $U_{x,y}AU_{x,y}\subset O_{x,y}(A)$. The continuity of the group operation at $1_G$ yields an open neighborhood $V_{x,y}\subset G$ of $1_G$ such that $V_{x,y}\cdot V_{x,y}\subset U_{x,y}$. By the compactness of the space $G\times G$ the open cover $\{xV_{x,y}\times V_{x,y}y:(x,y)\in G\times G\}$ of $G\times G$ has a finite subcover $\{xV_{x,y}\times V_{x,y}y:(x,y)\in F\}$ where $F$ is a finite subset of $G\times G$. Consider the open neighborhood $V=\bigcap_{(x,y)\in F}V_{x,y}$ of $1_G$ and the open neighborhood $V\kern-1pt AV$ of the closed set $A$. By the Urysohn Lemma, there is a continuous function $f:G\to[0,1]$ such that $f(A)\subset \{0\}$ and $f(G\setminus V\kern-1pt AV)\subset \{1\}$. By the additivity of the Haar measure $\lambda$, there is a number $t\in(0,1)$ whose preimage $f^{-1}(t)$ has Haar measure $\lambda(f^{-1}(t))=0$. In this case the open neighborhood $W=f^{-1}\big([0,t)\big)\subset V\kern-1pt AV$ of $A$ has boundary $\partial W\subset f^{-1}(t)$ of Haar measure zero. By Corollary~\ref{zeroboundary}, $\iss_{213}(W)=\lambda(W)$.

We claim that $\mu(aWb)<\lambda(A)-\e$ for any points $a,b\in G$. Since $\{xV_{x,y}\times V_{x,y}y:(x,y)\in F\}$ is a cover of $G\times G$, there is a pair $(x,y)\in F$ such that $a\in xV_{x,y}$ and $b\in V_{x,y}y$. Then $$aWb\subset aVAVb\subset xV_{x,y}V\kern-1pt AVV_{x,y}y\subset xV_{x,y}V_{x,y}AV_{x,y}V_{x,y}y\subset xU_{x,y}AU_{x,y}y\subset xO_{x,y}(A)y$$and hence
$$\mu(aWb)\le\mu(xO_{x,y}(A)y)<\lambda(A)-\e.$$
By Corollary~\ref{zeroboundary},
$$\iss_{213}(W)\le\sup_{a,b\in G}\mu(aWb)\le\lambda(A)-\e<\lambda(W)=\iss_{213}(W),$$
which is a desired contradiction. So, $\iss_{213}(A)=\lambda(A)$.
\end{proof}

\begin{theorem}\label{Solecki=Haar} $\lambda_*(A)\le\iss_{213}(A)\le \lambda(\bar A)$ for each subset $A\subset G$.
\end{theorem}

\begin{proof}
By Lemma~\ref{closed} and the monotonicity of the Solecki submeasure, we get
$$\lambda_*(A)=\sup\{\lambda(F):F=\bar F\subset A\}=\sup\{\iss_{213}(F):F=\bar F\subset A\}\le\iss_{213}(A).$$The inequality $\iss_{213}(A)\le\lambda(\bar A)$ has been proved in Theorem~\ref{Bohr}.
\end{proof}

\begin{remark} Denote by
 $$\A_0=\{A\subset G:\lambda(\partial A)=0\}=\{A\subset G:\iss_{213}(\partial A)=0\}$$the algebra of sets whose boundary has zero Solecki submeasure. Theorem~\ref{Solecki=Haar} implies that the restrictions of the Haar measure and the Solecki submeasure to the algbera $\A_0$ coincide. Since $\A_0$ generates the $\sigma$-algebra $\mathcal B(G)$ of Borel subsets of $G$, this implies that the Solecki submeasure completely determines the Haar measure! So, the Haar measure has an essential algebraic component!
\end{remark}

Both inequalities $\lambda_*(A)\le\iss_{213}(A)\le\lambda(\bar A)$ can be strict.

To show that the lower bound can be strict, we shall prove that $\lambda(A^\bullet)\le\iss_{213}(A)$.
Here $A^\bullet$ is the largest open subset of $G$ such that $A^\bullet\setminus A$ is meager. It is clear that $A^\circ\subset A^\bullet\subset\bar A$.

\begin{lemma} Each subset $A\subset G$ has $\lambda(A^\bullet)\le\iss_{213}(A)$.
\end{lemma}

\begin{proof} Assume conversely that $\iss_{213}(A)<\lambda(A^\bullet)$ and put $\e=\frac12(\lambda(A^\bullet)-\iss_{213}(A))$. Since $\uss_{213}(A)=\iss_{213}(A)<\lambda(A^\bullet)-\e$, there is a finite subset $F\subset G$ such that $\sup_{x,y\in G}|xFy\cap A|/|F|<(\lambda(A^\bullet)-\e)$.
By the regularity of the Haar measure $\lambda$, some compact set $K\subset A^\bullet$ has Haar measure $\lambda(K)>\lambda(A^\bullet)-\e$. By Theorem~\ref{Solecki=Haar}, $\lambda(K)=\iss_{213}(K)\le\max_{x,y\in G}|xFy\cap K|/|F|$. So, there are points $u,v\in G$ such that $|uFv\cap A^\bullet|\ge |uFv\cap K|\ge \lambda(K)\cdot|F|$. Let $T=\{t\in F:utv\in A^\bullet\}$ and observe that $|T|=|uFv\cap A^\bullet|\ge\lambda(K)\cdot|F|$. For every $t\in T$ consider the homeomorphism $s_t:G\to G$, $s_t:x\mapsto xtv$, and observe that $s_t^{-1}(A^\bullet)$ is an open neighborhood of the point $u$. Since the set $A^\bullet\setminus A$ is meager in $G$ its preimage $s_t^{-1}(A^\bullet\setminus A)$ is a meager set in $G$. Since the space $G$ is compact and hence Baire, in the open neighborhood $V_u=\bigcap_{t\in T}s_t^{-1}(A^\bullet)$ of the point $u$ we can find a point $x\in V_u$ which does not belong to the meager set $\bigcup_{t\in T}s_t^{-1}(A^\bullet\setminus A)$. For this point $x$ we get $s_t(x)\in A$ for all $t\in T$, which implies that $xTv\subset A$ and then $|xFv\cap A|\ge|xTv\cap A|=|xTv|=|T|\ge\lambda(K)\cdot|F|>(\lambda(A^\bullet)-\e)\cdot|F|$, which contradicts the choice of $F$.
\end{proof}

\begin{corollary} Any subset $A$ of a compact topological group $G$ has
 $$\max\{\lambda_*(A),\lambda(A^\bullet)\}\le\iss_{213}(A)\le\lambda(\bar A).$$
\end{corollary}

\begin{example}\label{ex11.9} The compact abelian group $\IT=\{z\in \IC:|z|=1\}$
contains a Borel subset $A$ such that
$$\frac14=\lambda(A)=\lambda(A^\bullet)=\iss_{213}(A)<\lambda(A\cup A^\bullet)=\lambda(\bar A)=\frac12.$$
\end{example}

\begin{proof} Consider the open subset $U=\{e^{i\varphi}:0<\varphi<\pi/2\}\subset\IT$ of Haar measure $\lambda(U)=1/4$  and the countable dense subset  $Q=\{e^{i\varphi}:\varphi\in\pi\cdot\mathbb Q\}$ where $\IQ$ is the set of rational numbers. By the regularity of the Haar measure $\lambda$ on $\IT$, the set $U\setminus Q$ contains a $\sigma$-compact (meager) subset $K$ of Haar measure $\lambda(K)=\lambda(U\setminus Q)=\frac14$. Now consider the set $A=(U\setminus K)\cup (-K)$ where $-K=\{-z:z\in K\}$. The finite set $F=\{1,-1,i,-i\}$ witnesses that $\iss_{213}\le\sup_{x,y\in\IT}|xFy\cap A|/|F|=\frac14$.
It follows that, $A^\bullet=U$ and thus
$$
\frac14=\lambda(A)=\lambda(A^\bullet)\le\iss_{213}(A)\le\frac14.$$
On the other hand,
$$\lambda(A\cup A^\bullet)=\lambda(U\cup(-K))=\frac14+\frac14=\frac12=\lambda(\bar U\cup(-\bar U))=\lambda(\bar A).$$
\end{proof}

\section{The extremal submeasure $\sis_{123}$}

On each $G$-space $X$ the extremal submeasure $\sis_{123}:\mathcal P(X)\to[0,1]$ is defined by the formula
$$\sis_{123}(A)=\sup_{\mu_1\in P_\w(G)}\inf_{\mu_2\in P_\w(G)}\sup_{\mu_3\in P_\w(X)}\mu_1*\mu_2*\mu_3(A).$$
The definition implies that $\sis_{123}$ is a $G$-invariant density on $X$ and
$$\sis_{123}(A)=\sup_{\mu_1\in P_\w(G)}\inf_{\mu_2\in P_\w(G)}\sup_{x\in X}\mu_1*\mu_2*\delta_x(A)$$
for every $A\subset X$.

\begin{proposition} On each $G$-space $X$ the invariant density $\sis_{123}$ is subadditive.
\end{proposition}

\begin{proof} It suffices to check that $\sis_{123}(A\cup B)\le\sis_{123}(A)+\sis_{123}(B)+2\e$ for every subsets $A,B\subset X$ and real number $\e>0$. This will follow as soon as for any measure $\mu_1\in P_\w(G)$ we find a measure $\mu_2\in P_\w(G)$ such that $\sup_{\mu_3\in P_\w(X)}\mu_1*\mu_2*\mu_3(A\cup B))<\sis_{123}(A)+\sis_{123}(B)+2\e$.

By the definition of $\sis_{123}(A)$, for the measure $\mu_1$ there is a measure $\nu_2\in P_\w(G)$ such that
$$\sup_{\nu_3\in P_\w(X)}\mu_1*\nu_2*\mu_3(A)<\sis_{123}(A)+\e.$$ By the definition of $\sis_{123}(B)$ for the measure $\eta_1=\mu_1*\nu_2$ there is a measure $\eta_2\in P_\w(G)$ such that $$\sup_{\eta_3\in P_\w(X)}\eta_1*\eta_2*\eta_3(B)<\sis_{123}(B)+\e.$$ We claim that the measure $\mu_2=\nu_2*\eta_2$ has the required property. Indeed, for every measure $\mu_3\in P_\w(X)$ we get
$$\mu_1*\mu_2*\mu_3(A\cup B)\le \mu_1*\nu_2*(\eta_2*\mu_3)(A)+(\mu_1*\nu_2)*\eta_2*\mu_3(B)<\sis_{123}(A)+\e+\sis_{123}(B)+\e.$$
\end{proof}

\begin{proposition} For each $G$-space $X$
 $$\is_{12}\le\his_{12}\le\sis_{123}\le\iss_{213}.$$
\end{proposition}

\begin{proof} Exercise.
\end{proof}

\begin{proposition} For each amenable group $G$ and  a $G$-space $X$ we get $d^*=\is_{12}=\sis_{123}$.
\end{proposition}

\begin{proof} The equality $d^*=\is_{12}$ has been proved in Theorem~\ref{d=is}. Assuming that $\is_{12}\ne\sis_{123}$, we can find a set $A\subset X$ such that $\is_{12}(A)<\sis_{123}(A)-3\e$ for some $\e>0$. By the definition of $\is_{12}(A)$, there is a measure $\mu_1\in P_\w(G)$ such that $\sup_{\mu_2\in P_\w(X)}\mu_1*\mu_2(A)<\is_{12}(A)+\e$. By the definition of $\sis_{123}(A)$, there is a measure $\nu_1\in P_\w(G)$ such that $\inf_{\nu_2\in P_\w(G)}\sup_{\nu_3\in P_\w(X)}\nu_1*\nu_2*\nu_3(A)>\sis_{123}(A)-\e$. The Emerson's characterization of amenability \cite{Emerson} implies the existence of two measures $\mu_2,\nu_2\in P_\w(G)$ such that $\|\mu_1*\mu_2-\nu_1*\nu_2\|<\e$ where $\|\cdot\|$ is the $\ell_1$-norm of the dual Banach space $(\ell_\infty(G))^*\supset\ell_1(G)\supset P_\w(G)$. By the choice of the measure $\nu_1$, for the masure $\nu_2$ there is a measure $\nu_3\in P_\w(X)$ such that $\nu_1*\nu_2*\nu_3(A)>\sis_{123}(A)-\e$.
The choice of the measure $\mu_1$ guarantees that $\mu_1*\mu_2*\nu_3(A)<\is_{12}(A)+\e$.
Then $$
\begin{aligned}
\sis_{123}(A)-\e & <\nu_1*\nu_2*\nu_3(A)= \mu_1*\mu_2*\nu_3(A)+ (\nu_1*\nu_2*\nu_3(A)-\mu_1*\mu_2*\nu_3(A))\le\\
&\le\is_{12}(A)+\e+\|\nu_1*\nu_2*\nu_3-\mu_1*\mu_2*\nu_3\|\le\\
&\le \is_{12}(A)+\e+\|\nu_1*\nu_2-\mu_1*\mu_2\|<\is_{12}(A)+2\e,
\end{aligned}
$$which contradicts $\is_{12}(A)<\sis_{123}(A)-3\e$.
\end{proof}

\begin{problem} Is $\is_{12}=\sis_{123}$ for each amenable $G$-space $X$?
\end{problem}

The equality $d^*=\sis_{123}$ holding for each $G$-space $X$ with amenable group $G$, suggests to call the submeasure $\sis_{123}$ the {\em upper Banach submeasure} on $X$.

\begin{theorem}\label{t9.2} If a subset $A$ of a group $G$ has positive upper Banach submeasure
$\sis_{123}(A)>0$, then for some finite set $E\subset G$ the set $(A^{-1}A)^{\wr E}=\bigcup_{x\in E}x^{-1}A^{-1}Ax$ has $$\cov((AA^{-1})^{\wr E})\le 1/\sis_{123}(A).$$
\end{theorem}

\begin{proof} Fix $\e>0$ so small that each integer number $n\le\frac1{\sis_{123}(A)-\e}$ is $\le\frac1{\sis_{123}(A)}$. By the definition of the submeasure $\sis_{123}(A)$, there is a measure $\mu_1\in P_\w(G)$ such that $\inf_{\mu_2\in P_\w(G)}\sup_{\mu_3\in P_\w(G)}\mu_1*\mu_2*\mu_3(A)>\sis_{123}(A)-\e$. Write $\mu_1$ as the convex combination $\mu_1=\sum_{i=1}^n\alpha_i\delta_{a_i}$ and put $E=\{a_1,\dots,a_n\}$.

Using Zorn's Lemma, choose a maximal subset $M\subset G$ such that for every $a\in E$ the indexed family $(xa^{-1}A)_{x\in M}$ is disjoint. By the maximality of $M$, for every point $g\in G$ there are points $x\in M$ and $a\in E$ such that $ga^{-1}A\cap xa^{-1}A\ne\emptyset$ and hence $g\in xa^{-1}AA^{-1}a\subset M(AA^{-1})^{\wr E}$.
So, $G=M(AA^{-1})^{\wr E}$ and hence $\cov((AA^{-1})^{\wr E})\le |M|$. To complete the proof, it remains to check that the set $M$ has cardinality $|M|\le 1/(\sis_{123}(A)+\e)$.

Assuming the opposite, we could find a finite subset $F\subset M$ of cardinality $|F|>1/(\sis_{123}(A)-\e)$.
Consider the measure $\mu_2=\frac1{|F|}\sum_{x\in F}\delta_{x^{-1}}$. For this measure there is a measure $\mu_3\in P_\w(G)$ such that $\mu_1*\mu_2*\mu_3(A)>\sis_{123}(A)-\e$. The measure $\mu_3$ can be assumed to be a Dirac measure $\mu_3=\delta_{y^{-1}}$ for some $y\in G$. Then $\mu_1*\mu_2(Ay)=\mu_1*\mu_2*\delta_{y^{-1}}(A)>\sis_{123}(A)-\e$.

On the other hand, for every $i\in\{1,\dots,n\}$, the disjointness of the family $(xa_i^{-1}A)_{x\in F}$ implies that $\sum_{x\in F}\chi_{xa_i^{-1}A}(y^{-1})\le 1$ and then
$$\sis_{123}(A)-\e<\mu_1*\mu_2(Ay)=\sum_{i=1}^n\alpha_i\sum_{x\in F}\frac1{|F|}\delta_{a_ix^{-1}}(Ay)=\frac1{|F|}\sum_{i=1}^n\alpha_i\sum_{x\in F}\chi_{xa_i^{-1}A}(y^{-1})\le\frac1{|F|},$$which contradicts the choice of $F$.
 \end{proof}

\begin{corollary} If a subset $A$ of a group $G$ has $\sis_{123}(A)>0$, then $\cov(AA^{-1}AA^{-1})<\infty$.
\end{corollary}

\begin{proof} By Theorem~\ref{t9.2}, there is a finite set $F\subset G$ such that $G=\bigcup_{x,y\in F}xAA^{-1}y$. By Theorem~\ref{slob}, there are points $x,y\in F$ and a finite set $E\subset G$ such that
 $G=E(xAA^{-1}y)(xAA^{-1}y)^{-1}$, which implies that $G=ExAA^{-1}AA^{-1}x^{-1}$ and $G=ExAA^{-1}AA^{-1}$.
So, $\cov(AA^{-1}AA^{-1})\le|Ex|<\infty$.
\end{proof}

\begin{corollary} A subgroup $H$ of a group $G$ has infinite index in $G$ if and only if $\sis_{123}(H)=0$.
\end{corollary}

\begin{corollary} Each subset $A$ of cardinality $|A|<|G|$ in an infinite grup $G$ has submeasure $\sis_{123}(A)=0$.
\end{corollary}

This corollary contrasts with the strange property of the Solecki submeasure $\iss_{213}$ described in Example~\ref{strange}.

The subadditivity of the submeasure $\sis_{123}$ and Theorem~\ref{t9.2} imply the following corollary giving a partial answer to Protasov's Problem.

\begin{corollary}\label{c9.3} For any partition $G=A_1\cup\dots\cup A_n$ of a group $G$ some cell $A_i$ of the partition has $\cov\big((A^{-1}A)^{\wr E}\big)\le n$ for some finite set $E\subset G$.
\end{corollary}

\begin{corollary} If an (analytic) subset $A$ of a Polish group $G$ has positive submeasure $\sis_{123}(A)>0$, then
the set $AA^{-1}$ is not meager (and $AA^{-1}AA^{-1}$ is a neighborhood of the unit in $G$).
\end{corollary}

\section{Concluding Remarks and an Open Problem}

The extremal densities $\is_{12}$, $\Si_{21}$, $\sis_{123}$ are initial representatives of the hierarchy of extremal densities defined on each group $G$ as follows.

 Given a positive integer number $n\in\IN$, a permutation $s:\{1,\dots,n\}\to\{1,\dots,n\}$, and
a function $\mathsf e:\{1,\dots,n\}\to\{\mathsf{i},\mathsf{s},\mathsf{u},\mathsf{I},\mathsf{S}\}$ with
$\mathsf e(\{1,\dots,n\}\setminus\{s^{-1}(n)\})\subset\{\mathsf{u},\mathsf{i},\mathsf{s}\}$, define
the density $\mathsf e_s:\mathcal P(G)\to[0,1]$ by the formula
$$\mathsf e_s(A)={\mathsf e(1)}_{\mu_1}\cdots
{\mathsf e(n)}_{\mu_n}\mu_{s(1)}*\cdots*\mu_{s(n)}(A)\mbox{ \ for \ }A\subset G$$where
$\mathsf{u}_{\mu_i}$, $\mathsf{i}_{\mu_i}$, $\mathsf{I}_{\mu_i}$, $\mathsf{s}_{\mu_i}$, $\mathsf{S}_{\mu_i}$ stand for the operators $\inf_{\mu_i\in P_u(G)}$, $\inf_{\mu_i\in P_\w(G)}$,
$\inf_{\mu_i\in P(G)}$, $\sup_{\mu_i\in P_\w(G)}$, $\sup_{\mu_i\in P(G)}$, respectively.

The density $\mathsf e_s$ will be called the {\em extremal density} generated by the function $\mathsf e$ and the substitution $s$. To shorten the notations, we shall identify the functions $\mathsf e$ and $s$ with the sequences $(\mathsf e(1),\dots,\mathsf e(n))$ and $(s(1),\dots,s(m))$ or even words $\mathsf e(1)\cdots \mathsf e(n)$ and $s(1)\cdots s(m)$.

Observe that the simplest extremal densities
$\mathsf{i}_1$ and $\mathsf{s}_1$ can be calculated by the formulas
$$\mathsf{i}_1(A)=\begin{cases}0&\mbox{if $A\ne X$}\\
1&\mbox{if $A=X$}\end{cases}
\mbox{ \ \ and \ \ } \mathsf{s}_1(A)=\begin{cases}0&\mbox{if $A=\emptyset$}\\
1&\mbox{if $A\ne\emptyset$}\end{cases}$$
implying that $\mathsf{i}_1$ and $\mathsf{s}_1$ are the smallest and largest densities on $X$,  respectively. Therefore, the extremal densities $\mathsf{is}_{12}$ and $\mathsf{is}_{21}$ are the simplest nontrivial extremal densities in this hierarchy.
This suggests the following problem, or rather, a program of research.

\begin{problem} Study the properties of the extremal densities $\mathsf e_s$ on groups. Detect extremal densities which are subadditive. Study the interplay between various extremal densities on a group.
Find further applications of extremal densities in combinatorics of groups and $G$-spaces.
\end{problem}

\begin{observation} Each extremal density on an abelian group $G$ is equal to $\mathsf i_1$, $\mathsf s_1$, $\is_{12}$ or $\mathsf{si}_{12}$.
\end{observation}

\begin{question}[Kwietniak] Is the family of all extremal densities on each (amenable) group $G$ finite?
\end{question}

\section{Acknowledgments}

The author express his sincere thanks to Anna Pelczar-Barwacz and Piotr Niemeic for the invitation to teach this lecture course in Jagiellonian University, to Dominik Kwietniak for valuable comments and suggestions (which allowed to generalized some notions and results related to extremal measures from groups to $G$-spaces) and also to all Ph.D. students of Jagiellonian University for their active participation in the lectures.

The preprints \cite{Ban} and \cite{BPS} will be updated taking into account the progress made in this lecture notes.

\end{document}